\pdfoutput=1
\RequirePackage{ifpdf}
\ifpdf % We are running pdfTeX in pdf mode
\documentclass[pdftex]{sigma}
\else
\documentclass{sigma}
\fi

\renewcommand{\epsilon}{\varepsilon}
\newcommand\N{\mathbb N}
\newcommand\Z{\mathbb Z}

\newcommand{\cM}{\mathcal{M}}

\newtheorem{Theorem}{Theorem}
\newtheorem*{NNTheorem}{Theorem}
\newtheorem*{MConjecture}{Main Conjecture}

\newtheorem{Lemma}[Theorem]{Lemma}
\newtheorem{Corollary}[Theorem]{Corollary}
\newtheorem{Proposition}[Theorem]{Proposition}

\theoremstyle{definition}
\newtheorem{Remark}[Theorem]{Remark}

\begin{document}

\allowdisplaybreaks

\newcommand{\arXivNumber}{2004.02749}

\renewcommand{\thefootnote}{}

\renewcommand{\PaperNumber}{086}

\FirstPageHeading

\ShortArticleName{Uniform Lower Bound for Intersection Numbers of $\psi$-Classes}

\ArticleName{Uniform Lower Bound for Intersection Numbers\\ of $\boldsymbol{\psi}$-Classes\footnote{This paper is a~contribution to the Special Issue on Algebra, Topology, and Dynamics in Interaction in honor of Dmitry Fuchs. The full collection is available at \href{https://www.emis.de/journals/SIGMA/Fuchs.html}{https://www.emis.de/journals/SIGMA/Fuchs.html}}}

\Author{Vincent DELECROIX~$^{\dag^1}$, \'Elise GOUJARD~$^{\dag^2}$, Peter ZOGRAF~$^{\dag^3\dag^4}$ and Anton ZORICH~$^{\dag^5\dag^6}$}

\AuthorNameForHeading{V.~Delecroix, \'E.~Goujard, P.~Zograf and A.~Zorich}

\Address{$^{\dag^1}$~LaBRI, Domaine universitaire, 351 cours de la Lib\'eration, 33405 Talence, France}
\EmailDD{\href{mailto:20100.delecroix@gmail.com}{20100.delecroix@gmail.com}}

\Address{$^{\dag^2}$~Institut de Math\'ematiques de Bordeaux, Universit\'e de Bordeaux,\\
\hphantom{$^{\dag^2}$}~351 cours de la Lib\'eration, 33405 Talence, France}
\EmailDD{\href{mailto:email@address}{elise.goujard@gmail.com}}

\Address{$^{\dag^3}$~Steklov Mathematical Institute, Fontanka 27, St.~Petersburg 191023, Russia}
\EmailDD{\href{mailto:zograf@pdmi.ras.ru}{zograf@pdmi.ras.ru}}

\Address{$^{\dag^4}$~Chebyshev Laboratory, St.~Petersburg State University,\\
\hphantom{$^{\dag^4}$}~14th Line V.O.~29B, St.~Petersburg, 199178, Russia}

\Address{$^{\dag^5}$~Center for Advanced Studies, Skoltech, Russia}
\EmailDD{\href{mailto:anton.zorich@gmail.com}{anton.zorich@gmail.com}}

\Address{$^{\dag^6}$~Institut de Math\'ematiques de Jussieu -- Paris Rive Gauche, B\^atiment Sophie Germain,\\
\hphantom{$^{\dag^6}$}~Case 7012, 8 Place Aur\'elie Nemours, 75205 PARIS Cedex 13, France}

\ArticleDates{Received April 09, 2020, in final form August 21, 2020; Published online August 26, 2020}

\Abstract{We approximate intersection numbers $\big\langle \psi_1^{d_1}\cdots \psi_n^{d_n}\big\rangle_{g,n}$ on Deligne--Mumford's moduli space $\overline{\cM}_{g,n}$ of genus~$g$ stable complex curves with $n$ marked points by certain closed-form expressions in $d_1,\dots,d_n$. Conjecturally, these approximations become asymptotically exact uniformly in $d_i$ when $g\to\infty$ and $n$ remains bounded or grows slowly. In this note we prove a lower bound for the intersection numbers in terms of the above-mentioned approximating
expressions multiplied by an explicit factor $\lambda(g,n)$, which tends to $1$ when $g\to\infty$ and
$d_1+\dots+d_{n-2}=o(g)$.}


\Classification{14C17; 14H70}

\begin{flushright}
\begin{minipage}{60mm}
\it Dedicated to D.B.~Fuchs\\ on the occasion of his 80th birthday
\end{minipage}
\end{flushright}

\renewcommand{\thefootnote}{\arabic{footnote}}
\setcounter{footnote}{0}

\section{Introduction}\label{s:Introduction}

Let $\overline{\mathcal{M}}_{g,n}$ be the Deligne--Mumford
moduli space of genus $g$ complex stable algebraic curves
(possibly with nodes), with $n>0$ distinct labeled marked
points. Consider the tautological line bundles
$\mathcal{L}_i\to\overline{\mathcal{M}}_{g,n}$, $i=1,\dots,n$, defined fiberwise by
$\mathcal{L}_i|_{C,x_1,\dots,x_n} \cong
T^*_{x_i}C$, where $C$ is a~genus~$g$ curve with marked
points $x_1,\dots,x_n$ (the definition makes sense since
the marked points are not allowed to coincide with the nodes).

In the late 1980-ies E.~Witten \cite{Witten} introduced a
theory of two-dimensional topological gravity, where the
classes $\psi_i=c_1(\mathcal{L}_i)$, $i=1,\dots,n,$ played
the role of observables, and the intersection numbers
\begin{gather*}
\langle\tau_{d_1}\cdots\tau_{d_n}\rangle_{g,n}
=\big\langle\psi_1^{d_1}\cdots\psi_n^{d_n}\big\rangle_{g,n}
=\int_{\overline{\mathcal{M}}_{g,n}}
\psi_1^{d_1}\cdots\psi_n^{d_n} ,
\end{gather*}
where $d_1+\dots +d_n=3g-3+n$, represented correlators of
the theory. Following a common convention, we will omit $g$
and $n$, or just $n$ when they are clear from the context.
On the basis of a low genus evidence and considerations
from physics, E.~Witten conjectured that the generating function
\begin{gather*}
F(t_0,t_1,\dots)
=\sum_{\substack{g\geq 0\\n\geq 1}}\frac{1}{n!}
\sum_{\substack{d_1+\cdots+ d_n\\=3g-3+n}}
\langle\tau_{d_1}\cdots\tau_{d_n}\rangle
t_{d_1}\cdots t_{d_n}
\end{gather*}
(total free energy of Witten's two-dimensional topological
gravity) satisfies the KdV (Korteveg--de Vries) hierarchy. An
equivalent formulation of Witten's conjecture is that the
partition function $e^F$ is a highest weight vector of a
Virasoro algebra representation, see papers~\cite{Dij, DVV}
of R.~Dijkgraaf and of R.~Dijkgraaf, H.~Verlinde, and E.~Verlinde. Witten's conjecture was first proven by
M.~Kontsevich \cite{Kontsevich}, and then several other
proofs appeared, including the proofs due to
A.~Okounkov and R.~Pandharipande~\cite{Okounkov:Pandharipande},
to M.~Kazarian and S.~Lando~\cite{Kazarian:Lando},
to M.~Mirzakhani~\cite{Mirzakhani:simple:geodesics:and:volumes,
Mirzakhani:volumes:and:intersection:theory}. Note that the
two-dimensional topological gravity may be interpreted as
the simplest instance of the Gromov--Witten theory, where
the tagret space is a point.

Computability of the intersection numbers
$\langle\tau_{d_1}\cdots\tau_{d_n}\rangle$ is an important
open problem. Besides the cases of small $g$ or $n$ no
general explicit formulas for the numbers
$\langle\tau_{d_1}\cdots\tau_{d_n}\rangle_{g,n}$ are known.
In applications it is often sufficient to know large genus
behavior of these intersection numbers, and here either not
much was known in general before the current paper
and before even more recent beautiful
work by A.~Aggarwal~\cite{Aggarwal:intersection:numbers}.

In this note we prove a simple lower bound for the numbers
that appears to be asymptotically exact as $g\to\infty$. To
begin with, we recall some basic facts about the numbers
$\langle\tau_{d_1}\cdots\tau_{d_n}\rangle$ (Witten's
correlators). They are uniquely defined by the initial data
\[
\big\langle \tau_0^3 \big\rangle = 1,\qquad
\langle \tau_{1} \rangle = \frac{1}{24}
\]
via the recursive relations known as Virasoro constraints that we present below.

\textbf{Virasoro constraints} (in Dijkgraaf--Verlinde--Verlinde form~\cite{Dij, DVV}):
\begin{gather}
\langle\tau_{k+1}\tau_{d_1}\cdots\tau_{d_n}\rangle_g=\frac{1}{(2k+3)!!}\Bigg[\sum_{j=1}^n
\frac{(2k+2d_j+1)!!}{(2d_j-1)!!}\langle\tau_{d_1}\cdots
\tau_{d_{j}+k}\cdots\tau_{d_n}\rangle_g\nonumber\\
\hphantom{\langle\tau_{k+1}\tau_{d_1}\cdots\tau_{d_n}\rangle_g=}{} +\frac{1}{2}\sum_{\substack{r+s=k-1\\r,s\ge 0}}
(2r+1)!!(2s+1)!!\langle\tau_r\tau_s\tau_{d_1}\cdots\tau_{d_n}\rangle_{g-1}\nonumber\\
\hphantom{\langle\tau_{k+1}\tau_{d_1}\cdots\tau_{d_n}\rangle_g=}{} +\frac{1}{2}\sum_{\substack{r+s=k-1\\r,s\ge 0}} (2r+1)!!(2s+1)!!\nonumber\\
\hphantom{\langle\tau_{k+1}\tau_{d_1}\cdots\tau_{d_n}\rangle_g=}{}\times
\sum_{\{1,\dots,n\}=I\coprod J}\bigg\langle\tau_r\prod_{i\in
I}\tau_{d_i}\bigg\rangle_{g'}\bigg\langle\tau_s\prod_{i\in
J}\tau_{d_i}\bigg\rangle_{g-g'}\Bigg].\label{eq:virasoro}
\end{gather}

For $k=-1$ and $k=0$ the above relations have particularly simple form.

\textbf{String equation ($\boldsymbol{k=-1}$):}
\begin{gather}\label{eq:string}
\langle \tau_0\tau_{d_1} \cdots \tau_{d_n}\rangle_{g,n+1}
=\langle \tau_{d_1-1} \cdots \tau_{d_n}\rangle_{g,n}
+\dots
+\langle \tau_{d_1} \cdots \tau_{d_n-1}\rangle_{g,n} .
\end{gather}

\textbf{Dilaton equation ($\boldsymbol{k=0}$):}
\begin{gather*}%\label{eq:dilaton}
\langle \tau_1\tau_{d_1} \cdots \tau_{d_n}\rangle_{g,n+1}
=(2g-2+n)\langle \tau_{d_1} \cdots \tau_{d_n}\rangle_{g,n} .
\end{gather*}

For any partition $\boldsymbol{d}$ of $3g-3+n$
into a sum of $n$ nonnegative integers define
$\epsilon(\boldsymbol{d})$ through the following equation:
\begin{gather*}%\label{eq:ansatz}
\big\langle \psi_1^{d_1} \cdots \psi_n^{d_n}\big\rangle_{g,n}
=
\frac{(6g-5+2n)!!}{(2d_1+1)!!\cdots(2d_n+1)!!}
\cdot\frac{1}{g!\cdot 24^g} \cdot\big(1+\epsilon(\boldsymbol{d})\big) .
\end{gather*}

We denote by $\Pi(m,n)$ the
set of ordered partitions of an integer $m$ into a sum of $n$ nonnegative integers.

\begin{MConjecture}[\cite{DGZZ:volume}]
For any strictly positive constant $C$
\[
\lim_{g\to+\infty}
\max_{1\le n\le C\log(g)}
\max_{\boldsymbol{d}\in\Pi(3g-3+n,n)}
|\epsilon(\boldsymbol{d})|=0 .
\]
\end{MConjecture}

Corollary~\ref{cor:inf:0} below makes the
first step towards a proof of the Main Conjecture. It
establishes an efficient uniform lower bound for
$\epsilon(\boldsymbol{d})$ for those partitions
$\boldsymbol{d}$ for which the sum of the first $n-2$
entries is small with respect to the sum of the remaining
two entries.

\begin{Remark}\label{rm:permutation:of:d}
It follows from the definition of $\epsilon(\boldsymbol{d})$ that $\epsilon(\boldsymbol{d})$ does not change under any permutation of the entries of $\boldsymbol{d}$.
\end{Remark}

\begin{Remark}\label{rm:alpha:1:2}
It is plausible, that much stronger statement might be
true, where the bound $n<C\log(g)$ is replaced by the bound
$n<g^\alpha$ with any fixed $\alpha$ satisfying
$\alpha<\frac{1}{2}$. The reason why one
cannot go beyond the bound $n<\sqrt{g}$ is explained at the
very end of Section~\ref{s:Introduction}.
\end{Remark}

\begin{Remark}[added in proofs]\label{rm:Aggarwal}
The Main Conjecture was proved in a very strong form in the
recent paper~\cite{Aggarwal:intersection:numbers} of
A.~Aggarwal.
\end{Remark}

\textbf{Motivation.}
Certain universality phenomena in flat and hyperbolic
geometry and in dynamics of surfaces manifest themselves in
large genera. The large genus asymptotics of the
Masur--Veech volumes of strata in moduli spaces of Abelian
differentials conjectured in~\cite{Eskin:Zorich} was
successfully proved by independent methods
in~\cite{Aggarwal:Volumes} and
in~\cite{Chen:Moeller:Sauvaget:Zagier}.
However, the analogous conjectures stated
in~\cite{ADGZZ:conjecture} on the large genus asymptotics
of the Masur--Veech volumes of strata in moduli spaces of
quadratic differentials are open.

There are several approaches to evaluation of Masur--Veech
volumes of such strata. The original approach of A.~Eskin
and A.~Okounkov discovered in~\cite{Eskin:Okounkov:pillowcase}
uses
characters of the symmetric group. Using modern computers
for exact computations based on this approach, \'E.~Goujard
evaluated in~\cite{Goujard:volumes} the volumes of all
strata up to dimension~11. Currently it is not known how to
obtain volume asymptotics based on this approach.

The recent paper~\cite{Chen:Moeller:Sauvaget} of D.~Chen,
M.~M\"oller and A.~Sauvaget expressed the Masur--Veech
volume of the principal stratum through certain very
special linear Hodge integrals. The papers~\cite{Kazarian}
of M.~Kazarian and~\cite{Yang:Zagier:Zhang} of D.~Yang,
D.~Zagier and Y.~Zhang provide extremely efficient
recursive formulae for these Hodge integrals, which allow
to compute exact values of the volumes of the principal
strata up to genus 250 and more. However, currently it is
not known how to prove large genus
asymptotic formulae for the Masur--Veech volume
of the principal stratum developing
this approach.

One more approach to evaluation of Masur--Veech volumes is
elaborated in our paper~\cite{DGZZ:volume} where we express
the Masur--Veech volume of the principal stratum of
meromorphic quadratic differentials with at most simple
poles through intersection numbers of $\psi$-classes and
suggest conjectures aimed to prove large genus
asymptotics of these volumes. This conjectural scheme
involves the Main Conjecture stated above as one of the key
ingredients.

\begin{Remark}[added in proofs]\label{rm:Added:in:proofs}
The asymptotic formula for the Masur--Veech volume
conjectured in~\cite{DGZZ:volume} was recently proved
in~\cite{Aggarwal:intersection:numbers} by A.~Aggarwal.
Together with the Main Conjecture proved by
A.~Aggarwal in the same paper, this
allowed us to provide in~\cite{DGZZ:random:multicurves} a
detailed description of the asymptotic geometry of random
square-tiled surfaces and of random simple closed
multicurves on surfaces of large genus.
\end{Remark}

\subsection{State of the art}\label{ss:state:of:the:art}

Currently we have the following evidence towards the Main Conjecture. Direct computation shows that $\epsilon(0,0,0)=0$. It is known~\cite{Witten} that
\[
\langle\tau_{3g-2}\rangle_{g,1}=\frac{1}{24^g\cdot g!} ,
\]
so for all $1$-correlators we have
\begin{gather}\label{eq:epsilon:for:1:correlators}
\epsilon(3g-2)=0 .
\end{gather}

Applying the string equation recursively we get
\begin{gather}
\label{eq:almost:zero:partition}
\big\langle\tau_0^{n-1}\tau_{3g-3+n}\big\rangle_{g,n}=\frac{1}{24^g\cdot g!} ,
\end{gather}
so for all partitions with at most one nonzero entry we have
\begin{gather}\label{eq:epsilon:0:1corr}
\epsilon\big(0^{n-1},3g-3+n\big)=0 .
\end{gather}

For $2$-correlators the Main Conjecture is valid. Namely, by Remark~\ref{rm:permutation:of:d}, we have $\epsilon(d_1,d_2)= \epsilon(d_2,d_1)$ for any $(d_1,d_2)\in\Pi(3g-1,2)$. Thus, we may assume that $d_1<d_2$. We have already seen that $\epsilon(0,3g-1)=0$. For the remaining $2$-partitions we have the following bounds:

\begin{NNTheorem}[\cite{DGZZ:volume}]
For all $g\in\N$ and for all integer $k$ satisfying
$2\le k\le \frac{3g-1}{2}$ the following bounds are valid:
\begin{gather}\label{eq:main:bounds}
- \frac{2}{6g-1}=\epsilon(1,3g-2)< \epsilon(k,3g-1-k)< 0=\epsilon(0,3g-1) .
\end{gather}
\end{NNTheorem}

We performed a detailed analysis of $\epsilon(k,3g-1-k)$ in~\cite{DGZZ:volume}
based on~\cite{Zograf:2:correlators}. In particular, for large~$g$ the error term $\epsilon(k,3g-1-k)$ rapidly tends to $0$ when $k$ approaches $\frac{3g-1}{2}$, so the statement of the above theorem can be seriously strengthened, if needed.

It is easy to compute $\epsilon\big(\boldsymbol{d}\big)$
explicitly for those partitions where all but one entries
$d_1,\dots,d_{n-1}$ are equal to $0$ or $1$. Namely,
we first apply recursively
the dilaton equation eliminating all those entries of the
partition, which are equal to $1$, and then apply~\eqref{eq:almost:zero:partition}.
In particular,
\[
\big\langle\tau_1^{n-1}\tau_{3g-2}\big\rangle_{g,n} =(2g-3+n)(2g-4+n)\cdots(2g-1)\cdot\frac{1}{24^g\cdot g!} ,
\]
so
\begin{gather*}
1+\epsilon ((1,1,\dots,1,3g-2) )=1+\epsilon\big(\big(1^{n-1},3g-2\big)\big)
 \\
 \hphantom{1+\epsilon ((1,1,\dots,1,3g-2) )}{} =
(2g-3+n)(2g-4+n)\cdots(2g-1)\cdot\frac{3^{n-1}\cdot(6g-3)!!}{(6g-5+2n)!!} \\
 \hphantom{1+\epsilon ((1,1,\dots,1,3g-2) )}{} =
\frac{6g-3+3(n-2)}{6g-1+2(n-2)}\cdot\frac{6g-3+3(n-3)}{6g-1+2(n-3)}\cdots\frac{6g-3}{6g-1} .
\end{gather*}
This implies that
for any constant $\alpha$ satisfying $0<\alpha<\frac{1}{2}$
(respectively $\frac{1}{2}<\alpha$)
we have
\begin{gather*}
\lim_{g\to+\infty}\max_{1\le n \le g^\alpha}
\big|\epsilon\big(1^{n-1},3g-2\big)\big|=
0, \quad \text{when }0<\alpha<\frac{1}{2} ,
\\
\lim_{g\to+\infty}\inf_{n \ge g^\alpha}
\epsilon\big(1^{n-1},3g-2\big)=
+\infty, \quad \text{when }\frac{1}{2}<\alpha ,
\end{gather*}
which explains why the restriction $\alpha<\frac{1}{2}$ in Remark~\ref{rm:alpha:1:2} cannot be loosened.

\section{Uniform lower bound}\label{s:Uniform:lower:bound}

Given a real
number $L$ and integers $g\ge 1$ and $n\ge 3$, denote
by $\Pi_L(3g-3+n,n)$ the following subset
of ordered partitions:
\[
\Pi_L(3g-3+n,n)=\big\{\boldsymbol{d}\in\Pi(3g-3+n,n) \,\big|\, d_1+\dots+d_{n-2} \le L \big\} .
\]
For any nonnegative $L$ and any integer $g\ge 1$ we define
$\Pi_L(3g-2,1)=\Pi(3g-2,1)$ and
$\Pi_L(3g-1,2)=\Pi(3g-1,2)$.

Define the following function of integer arguments $g$, $L$, satisfying $g>L\ge 0$:
\begin{gather}
\label{eq:lambda}
\lambda(g,L)=\left(\prod_{i=0}^{L-1}\left(1-\frac{1}{6(g-i)+1}\right)\right)
\cdot\left(1-\frac{2}{6(g-L)-1}\right) ,
\end{gather}
where, by convention,
\begin{gather}
\label{eq:lambda:0}
\lambda(g,0)=\left(1-\frac{2}{6g-1}\right) .
\end{gather}

It follows from the definition of $\lambda(g,L)$ that $0<\lambda(g,L)<1$ for any $g>L\ge 0$.

\begin{Theorem}\label{th:Main:Theorem}
Let $g$, $L$ be nonnegative integers such that $g>L$. For any partition
$\boldsymbol{d}\in\Pi_{L}(3g-3+n,n)$ one has
\begin{gather}
\label{eq:1:plus:epsilon:samller:than:lambda}
\epsilon(\boldsymbol{d})\ge\lambda(g,L)-1 .
\end{gather}
\end{Theorem}

\begin{Corollary}\label{cor:inf:0}
Let $L(g)$, where $g=1,2,\dots,$ be any sequence of nonnegative integers such that $L(g)=o(g)$ as $g\to+\infty$. One has
\begin{gather}
\label{eq:main:lower:bound}
\lim_{g\to+\infty}\inf_{n\ge 1}\min_{\boldsymbol{d}\in\Pi_{L(g)}(3g-3+n,n)}\epsilon(\boldsymbol{d})=0 .
\end{gather}
\end{Corollary}

\begin{proof}[Proof of Corollary~\ref{cor:inf:0}]
Definition~\eqref{eq:lambda} of $\lambda(g,L)$ implies that for any sequence $L(g)$ of nonnegative integers satisfying $L(g)=o(g)$ as $g\to+\infty$ one has
\[
\lim_{g\to+\infty} \lambda(g,L(g))=1 .
\]
Now~\eqref{eq:main:lower:bound} follows from combination of~\eqref{eq:1:plus:epsilon:samller:than:lambda} and~\eqref{eq:epsilon:for:1:correlators}.
\end{proof}

\begin{Remark}Proposition~3.2 in~\cite{Liu:Xu} claims
that for any triple $(n,K,M)$ of positive integers one has
\[
\lim_{g\to+\infty}\max_{\boldsymbol{d}\in\Pi_K(3g-3+n,n)}|\epsilon(\boldsymbol{d})|=0
\]
under the additional requirement that $d_{n-1}\le M$.
\end{Remark}

We start by proving three Lemmas (corresponding to
the string and the dilaton equations, and to Virasoro constraints).
It would be useful to introduce the following notation.
Given $\boldsymbol{d}\in\Pi(3g-3+n,n)$ let
\begin{gather}\label{eq:floor:brackets}
\lfloor \tau_{d_1} \cdots \tau_{d_n}\rfloor_{g,n}=
\frac{(6g-5+2n)!!}{(2d_1+1)!!\cdots(2d_n+1)!!}
\cdot\frac{1}{g!\cdot 24^g} .
\end{gather}
By definition of $\epsilon(\boldsymbol{d})$ we have
\begin{gather}\label{eq:def:of:epsilon}
\langle \tau_{d_1} \cdots \tau_{d_n}\rangle_{g,n}
=
\lfloor \tau_{d_1} \cdots \tau_{d_n}\rfloor_{g,n}
\cdot\big(1+\epsilon(\boldsymbol{d})\big) .
\end{gather}

From now on we suppose that $g\ge 1$.

\begin{Lemma}\label{lm:string}
Let $\boldsymbol{d}\in\Pi(3g-2+n,n-k)$ such that
$d_j>0$ for $j=1,\dots,n-k$. We assume that $k\ge 0$ and $n-k>0$.
Define $\delta_{\rm string}(0^{k+1},\boldsymbol{d})$
by equation
\begin{gather}
\big\lfloor \tau_0^{k+1}\tau_{d_1} \cdots \tau_{d_{n-k}}\big\rfloor_{g,n+1}
\cdot\big(1+\delta_{\rm string}\big(0^{k+1},\boldsymbol{d}\big)\big)
\nonumber\\
\qquad{} =
\big\lfloor \tau_0^k\tau_{d_1-1} \cdots \tau_{d_{n-k}}\big\rfloor_{g,n}
+\dots +\big\lfloor \tau_0^k\tau_{d_1} \cdots \tau_{d_{n-k}-1}\big\rfloor_{g,n} .\label{eq:string:floor}
\end{gather}
Then
\[
\delta_{\rm string}\big(0^{k+1},\boldsymbol{d}\big)=\frac{n-k-1}{6g-3+2n} .
\]
In particular, for any $\boldsymbol{d}$ as above and for any
$k\ge 0$ we have
\begin{gather}\label{eq:delta:string:is:positive}
\delta_{\rm string}\big(0^{k+1},\boldsymbol{d}\big)\ge 0 .
\end{gather}
\end{Lemma}

\begin{proof}Dividing both sides of equation~\eqref{eq:string:floor} by
$\big\lfloor \tau_0^{k+1}\tau_{d_1} \dots \tau_{d_{n-k}}\big\rfloor_{g,n+1}$
and applying definition~\eqref{eq:floor:brackets}
to all terms involved in the right-hand side of the resulting
equation we get
\begin{gather*}
1+\delta_{\rm string}\big(0^{k+1},\boldsymbol{d}\big)
=\frac{(2d_1+1)+\dots+(2d_{n-k}+1)}{6g-3+2n}
 \\
 \hphantom{1+\delta_{\rm string}\big(0^{k+1},\boldsymbol{d}\big)}{} =
\frac{2(3g-2+n)+(n-k)}{6g-3+2n}
=\frac{6g-4+3n-k}{6g-3+2n}
=1+\frac{n-k-1}{6g-3+2n} .\!\!\!\tag*{\qed}
\end{gather*}\renewcommand{\qed}{}
\end{proof}

\begin{Corollary}\label{cor:epsilon:0:2corr}
For any $(d_1,d_2)\in\Pi(3g-2+n,2)$ and for any $n\in\N$
one has
\begin{gather}\label{eq:epsilon:0:2corr}
\epsilon\big(0^{n-1},d_1,d_2\big)\ge -\frac{2}{6g-1} .
\end{gather}
\end{Corollary}
\begin{proof}If one of $d_1$, $d_2$
is equal to zero, the statement for arbitrary $n$
follows from~\eqref{eq:epsilon:0:1corr}, so from now on
we assume
that both $d_1$, $d_2$ are strictly positive.
For $n=1$ the statement follows directly from~\eqref{eq:main:bounds}.
This serves us as a base of induction
in $n$. Suppose that for all $n=1,\dots,k$ the statement
is true. Let us prove it for $n=k+1$:
\begin{gather*}
\big\langle \tau_0^{k+1}\tau_{d_1}\tau_{d_2}\big\rangle_{g,k+3}
=\big\langle \tau_0^k\tau_{d_1-1}\tau_{d_2}\big\rangle_{g,k+2}
+\big\langle \tau_0^k\tau_{d_1}\tau_{d_2-1}\big\rangle_{g,k+2}\\
\hphantom{\big\langle \tau_0^{k+1}\tau_{d_1}\tau_{d_2}\big\rangle_{g,k+3}}{} \ge
\left(1-\frac{2}{6g-1}\right)\cdot
\big(\big\lfloor \tau_0^k\tau_{d_1-1}\tau_{d_2}\big\rfloor_{g,k+2}
+\big\lfloor \tau_0^k\tau_{d_1}\tau_{d_2-1}\big\rfloor_{g,k+2}
\big)\\
\hphantom{\big\langle \tau_0^{k+1}\tau_{d_1}\tau_{d_2}\big\rangle_{g,k+3}}{}
 =
\left(1-\frac{2}{6g-1}\right)
\cdot\big(1+\delta_{\rm string}\big(0^{k+1},d_1,d_2\big)\big)\cdot
\big\lfloor \tau_0^{k+1}\tau_{d_1}\tau_{d_2}\big\rfloor_{g,k+3}\\
\hphantom{\big\langle \tau_0^{k+1}\tau_{d_1}\tau_{d_2}\big\rangle_{g,k+3}}{} \ge
\left(1-\frac{2}{6g-1}\right)\cdot
\big\lfloor \tau_0^{k+1}\tau_{d_1}\tau_{d_2}\big\rfloor_{g,k+3} ,
\end{gather*}
where the first equality is the string equation;
inequality between the first and the second lines is
the assumption of the induction; the equality
between the second and the third line is
equation~\eqref{eq:string:floor}; the inequality
between the third and the forth line is
an implication of~\eqref{eq:delta:string:is:positive}
and of the fact that all the factors in both lines are
positive.
\end{proof}

\begin{Corollary}\label{cor:L:equals:0}
For any $g,n\in\N$ and for any
$\boldsymbol{d}\in\Pi_0(3g-3+n,n)$ one has
\begin{gather}\label{eq:epsilon:is:smaller:than:lambda:0}
1+\epsilon(\boldsymbol{d})\ge\lambda(g,0) .
\end{gather}
\end{Corollary}

\begin{proof}Recalling convention~\eqref{eq:lambda:0} for $\lambda(g,0)$ we conclude that
for $n=1$ inequality~\eqref{eq:epsilon:is:smaller:than:lambda:0}
follows from~\eqref{eq:epsilon:for:1:correlators};
for $n=2$ inequality~\eqref{eq:epsilon:is:smaller:than:lambda:0}
follows from~\eqref{eq:main:bounds};
for $n\ge 3$ inequality~\eqref{eq:epsilon:is:smaller:than:lambda:0}
corresponds to~\eqref{eq:epsilon:0:2corr}.
\end{proof}

\begin{Lemma}\label{lm:dilaton}
Let $\boldsymbol{d}\in\Pi(3g-3+n,n)$.
Define $\delta_{\rm dilaton}(1,\boldsymbol{d})$
by equation
\begin{gather}\label{eq:dilaton:floor}
\lfloor \tau_1\tau_{d_1} \cdots \tau_{d_n}\rfloor_{g,n+1}
\cdot\big(1+\delta_{\rm dilaton}(1,\boldsymbol{d})\big)
=(2g-2+n)
\lfloor \tau_{d_1} \cdots \tau_{d_n}\rfloor_{g,n} .
\end{gather}
Then
\[
\delta_{\rm dilaton}(1,\boldsymbol{d})
=\frac{n-3}{6g-3+2n} .
\]
In particular,
\begin{gather}\label{eq:delta:dilaton}
\delta_{\rm dilaton}(1,\boldsymbol{d})
\begin{cases}
\ge 0&\text{when }n\ge 3 ,
\\
= -\dfrac{1}{6g+1}
&\text{when }n=2 ,
\vspace{1mm}\\
= -\dfrac{2}{6g-1}
&\text{when }n=1 .
\end{cases}
\end{gather}
\end{Lemma}

\begin{proof}Dividing both sides of equation~\eqref{eq:dilaton:floor} by
$\lfloor \tau_1\tau_{d_1} \cdots \tau_{d_n}\rfloor_{g,n+1}$,
applying definition~\eqref{eq:floor:brackets}
and canceling common factors in the numerator and in the denominator
of the resulting expression we get
\[
1+\delta_{\rm dilaton}(1,\boldsymbol{d})
=\frac{3(2g-2+n)}{6g-3+2n}
=\frac{6g-6+3n}{6g-3+2n} .\tag*{\qed}
\]\renewcommand{\qed}{}
\end{proof}

\begin{Corollary}\label{cor:epsilon:L:1}
For any partition
$\boldsymbol{d}\in\Pi_1(3g-3+n,n)$ one has
\begin{gather}\label{eq:epsilon:L:1}
1+\epsilon(\boldsymbol{d})\ge
\left(1-\frac{1}{6g+1}\right)
\cdot\left(1-\frac{2}{6g-1}\right) .
\end{gather}
\end{Corollary}

\begin{proof}We have seen in~\eqref{eq:epsilon:for:1:correlators}
that for all $1$-correlators we have
$\epsilon(3g-2)=0$,
so for $n=1$ the statement is true.
For $n=2$ the statement is a direct implication
of equation~\eqref{eq:main:bounds}.
Suppose that $n\ge 3$.

By Remark~\ref{rm:permutation:of:d}, the quantity
$\epsilon(\boldsymbol{d})$ does not change under any permutation of the entries of $\boldsymbol{d}$. Thus, we can permute the first $n-2$ elements of the partition without affecting the value of $\epsilon(\boldsymbol{d})$, in particular, we can place them in the growing order. Since the sum of the first $n-2$
elements is less than or equal to $1$ either they are all
equal to $0$ or they form the sequence $(0,\dots,0,1)$
after such reordering. If they all are equal to $0$, the
statement follows from equation~\eqref{eq:epsilon:0:2corr}
from Corollary~\eqref{cor:epsilon:0:2corr}.

It remains to consider the case when $n\ge 3$ and when the first $n-2$ elements form a sequence $(0,\dots,0,1)$. We prove first the desired inequality for partitions of the form $(1,d_1,d_2)$:
\begin{gather*}
\langle \tau_1\tau_{d_1}\tau_{d_2}\rangle_{g,3}
=(2g)\cdot\langle \tau_{d_1} \tau_{d_2}\rangle_{g,2}
\ge\left(1-\frac{2}{6g-1}\right)\cdot
(2g)\lfloor \tau_{d_1}\tau_{d_2}\rfloor_{g,2}
\\
\hphantom{\langle \tau_1\tau_{d_1}\tau_{d_2}\rangle_{g,3}}{}
=\left(1-\frac{2}{6g-1}\right)\cdot
\lfloor \tau_1\tau_{d_1}\tau_{d_2}\rfloor_{g,3}
\cdot\big(1+\delta_{\rm dilaton}(1,d_1,d_2)\big)
\\
\hphantom{\langle \tau_1\tau_{d_1}\tau_{d_2}\rangle_{g,3}}{}
\ge\left(1-\frac{2}{6g-1}\right)\left(1-\frac{1}{6g+1}\right)
\cdot\lfloor \tau_1\tau_{d_1}\tau_{d_2}\rfloor_{g,3} .
\end{gather*}
Here the first equality is the dilaton equation;
the inequality which follows is
equation~\eqref{eq:main:bounds}; the equality
between the first two lines is equation~\eqref{eq:dilaton:floor}
and the inequality between the second and the third line
is based on equation~\eqref{eq:delta:dilaton}.

To complete the proof of Corollary~\ref{cor:epsilon:L:1} we prove it for partitions of the form
$\big(0^{k+1},1,d_1,d_2\big)$ by induction in $k\ge 0$. The proof follows line-by-line the proof of Corollary~\ref{cor:epsilon:0:2corr}.
\end{proof}

\begin{Lemma}\label{lm:Virasoro}
Let $\boldsymbol{d}\in\Pi(3g-3+n-k,n)$, where $k,n\in\N$.
Define $\delta_{\rm Virasoro}(k+1,\boldsymbol{d})$
by equation
\begin{gather}
\lfloor\tau_{k+1}\tau_{d_1}\cdots\tau_{d_n}\rfloor_g
\big(1+\delta_{\rm Virasoro}(k+1,\boldsymbol{d})\big)
\nonumber\\
\qquad{} =\frac{1}{(2k+3)!!}
\Bigg(\sum_{j=1}^n
\frac{(2k+2d_j+1)!!}{(2d_j-1)!!}\lfloor\tau_{d_1}\cdots
\tau_{d_{j}+k}\cdots\tau_{d_n}\rfloor_g\nonumber\\
\qquad\quad{} +\frac{1}{2}\sum_{\substack{r+s=k-1\\r,s\ge 0}}
(2r+1)!!(2s+1)!!\lfloor\tau_r\tau_s\tau_{d_1}\cdots\tau_{d_n}\rfloor_{g-1}\Bigg) .\label{eq:Virasoro:floor}
\end{gather}
Then
\begin{gather}\label{eq:delta:Virasoro}
\delta_{\rm Virasoro}(k+1,\boldsymbol{d})
=\frac{n-3}{6g-3+2n}-\frac{2k(2n-5)}{(6g-3+2n)(6g-5+2n)} .
\end{gather}
\end{Lemma}

\begin{proof}Dividing both sides of equation~\eqref{eq:Virasoro:floor} by
$\lfloor \tau_{k+1}\tau_{d_1} \cdots \tau_{d_n}\rfloor_{g,n+1}$,
applying definition~\eqref{eq:floor:brackets}
and canceling common factors in the numerator and in the denominator
of the resulting expression we get
\begin{gather*}
1+\delta_{\rm Virasoro}(k+1,\boldsymbol{d})\\
\qquad{} =
\frac{1}{6g-3+2n}\cdot\left(
\big((2d_1+1)+\dots+(2d_n+1)\big)
+\frac{1}{2}\cdot k\cdot \frac{24g}{6g-5+2n}\right) \\
 \qquad{} =\frac{1}{6g-3+2n}\cdot\left(
\big(6g-6+3n-2k\big)+2k-2k\cdot\frac{2n-5}{6g-5+2n}\right)\\
\qquad{} = 1+\frac{n-3}{6g-3+2n}-\frac{2k(2n-5)}{(6g-3+2n)(6g-5+2n)} .\tag*{\qed}
\end{gather*}\renewcommand{\qed}{}
\end{proof}

\begin{Remark}\label{rm:lower:bound:Virasoro}
In expression~\eqref{eq:Virasoro:floor} we ignored the third term in the Virasoro constraints. Since this third term is, clearly, positive,
this is suitable for getting a lower bound
instead of exact asymptotics.
It is widely believed that the third term
of Virasoro constraints becomes negligible in large genera. We expect that technique from~\cite{Aggarwal:Siegel:Veech} might be useful for replacing
the lower bound in~\eqref{eq:main:lower:bound} by the exact asymptotics
under strengthening restrictions on $\alpha$.
\end{Remark}

We shall need the following technical corollary of Lemma~\ref{lm:Virasoro}.

\begin{Corollary}\label{cor:delat:Virasoro:bound}
Let $k,n$ be integers satisfying $k\ge 0$, $n\ge 2$.
Let $\boldsymbol{d}$ be a partition $\boldsymbol{d}\in\Pi(3g-3+n-k,n)$, such that
$k+1\le d_j$ for $j=1,\dots,n-2$,
and $k+d_1+\dots+d_{n-2}\le\frac{3}{2}g$.
Then
\begin{gather}\label{eq:delat:Virasoro:bound}
\delta_{\rm Virasoro}(k+1,\boldsymbol{d})\ge -\frac{1}{6g+1} .
\end{gather}
\end{Corollary}

\begin{proof}Use expression~\eqref{eq:delta:Virasoro} for
$\delta_{\rm Virasoro}(k+1,\boldsymbol{d})$.
When $n=2$ we get
\[
\delta_{\rm Virasoro}(k+1,\boldsymbol{d})=-\frac{1}{6g+1}+\frac{2k}{(6g+1)(6g-1)}\ge-\frac{1}{6g+1} .
\]

Let $n\ge 3$. By assumption, $(k+1)\le d_j$ for $j=1,\dots,d_{n-2}$, so $(k+1)(n-2)\le \frac{3}{2}g$, and hence $2k(2n-5)< 6g$, which implies that
\begin{gather*}
\delta_{\rm Virasoro}((k+1,\boldsymbol{d}))=\frac{n-3}{6g-3+2n}-\frac{2k(2n-5)}{(6g-3+2n)(6g-5+2n)}\\
\hphantom{\delta_{\rm Virasoro}((k+1,\boldsymbol{d}))}{} \ge -\frac{2k(2n-5)}{(6g-3+2n)(6g-5+2n)}
\ge -\frac{6g}{(6g+3)(6g+1)}>-\frac{1}{6g+1} .\!\!\!\tag*{\qed}
\end{gather*}\renewcommand{\qed}{}
\end{proof}

Before passing to the step of induction, we recapitulate the properties of the function $\lambda(g,L)$ defined in~\eqref{eq:lambda}.
Recall that the arguments $g$, $L$ of $\lambda(g,L)$ are nonnegative integers
satisfying $g>L$. Inequalities~\eqref{eq:lambda:inequalities}--\eqref{eq:lambda:inequality:ters}
below follow from the definition of $\lambda(g,L)$. Each inequality is
applied to those ordered pairs $g$, $L$ for which the argument of $\lambda$ on
both sides of the inequality belongs to the domain of definition of~$\lambda$. We have
\begin{gather}\label{eq:lambda:inequalities}
1>\lambda(g+1,L)>\lambda(g,L)>\lambda(g,L+1)>0 ,
\end{gather}
and
\begin{gather}\label{eq:lambda:inequality:bis}
\left(1-\frac{1}{6g+1}\right)
\cdot\lambda(g-1,L-1)=\lambda(g,L) .
\end{gather}
Combining the latter two relations we get
\begin{gather}
\left(1-\frac{1}{6g+1}\right)\cdot\lambda(g,L)>\left(1-\frac{1}{6(g-1)+1}\right)\cdot\lambda(g,L)\nonumber\\
\hphantom{\left(1-\frac{1}{6g+1}\right)\cdot\lambda(g,L)}{} >\left(1-\frac{1}{6(g-1)+1}\right)
\cdot\lambda(g-1,L) =\lambda(g,L+1) .\label{eq:lambda:inequality:ters}
\end{gather}

\begin{Proposition}[step of induction] \label{eq:prop:step:of:induction}
Suppose that for some nonnegative integer $L_0$
the following uniform bound is valid:
for all integers $g$, $L$ satisfying $g> L_0$, $0\le L\le L_0$,
for all partitions $\boldsymbol{d}\in\Pi_L(3g-2+n,n+1)$,
where $n\ge 0$, one has
\[
1+\epsilon(\boldsymbol{d})\ge \lambda(g,L) .
\]
Then for all integers $g$, $L$
satisfying $g> L_0+1$,
$0\le L\le L_0+1$,
for all partitions $\boldsymbol{d}\in\Pi_L(3g-2+n,n+1)$,
where $n\ge 0$, one also has
\[
1+\epsilon(\boldsymbol{d})\ge \lambda(g,L) .
\]
\end{Proposition}

\begin{proof}
We warn the reader that the total number of elements of the
partition is denoted in
Proposition~\ref{eq:prop:step:of:induction} by $n+1$ and
not by $n$ as in Theorem~\ref{th:Main:Theorem}. This allows
us to use formulae from the key Lemmas~\ref{lm:dilaton}
and~\ref{lm:Virasoro} without adjustments.

By convention $\Pi_L(3g-2,1)=\Pi(3g-2,1)$ and
$\Pi_L(3g-1,2)=\Pi(3g-1,2)$ for any $L\in\Z_{\ge 0}$.
We have seen in~\eqref{eq:epsilon:for:1:correlators}
that for all $1$-correlators we have
$\epsilon(3g-2)=0$,
so for $n=0$ the statement is trivially true.
For $n=1$ the statement is a direct implication
of inequality~\eqref{eq:main:bounds}:
\[
1+\epsilon(\boldsymbol{d})
\ge \left(1-\frac{2}{6g-1}\right)
=\lambda(g,0)
\ge\lambda(g,L) .
\]
Thus, from now on we can assume that $n\ge 2$. Let $\boldsymbol{d}\in\Pi_L(3g-2+n,n+1)$. If $L\le L_0$, then the statement makes part of the induction assumption. Hence, from now on we can assume that
\[
\boldsymbol{d}\in\Pi_{L_0+1}(3g-2+n,n+1)\setminus\Pi_{L_0}(3g-2+n,n+1) ,
\]
where $n\ge 2$ and $g> L_0+1$. This implies that
$d_1+\dots+d_{n-1}=L_0+1$ and, hence,
\begin{gather*}%\label{eq:sum:of:last:two}
d_{n}+d_{n+1}=3g-2+n-(L_0+1)> 2L_0+2 .
\end{gather*}

By Remark~\ref{rm:permutation:of:d}, the quantity $\epsilon(\boldsymbol{d})$ does not change under any permutation of the entries of~$\boldsymbol{d}$. Place to the leftmost position the smallest strictly positive element among the first $n-1$ elements. This operation does not change the last two elements of the partition and does not change the sum of its first $n-1$ elements. Denote the resulting partition by $(k+1,d_1,\dots,d_n)$. To prove the proposition we have to prove the inequality
\begin{gather}\label{eq:desired:inequality}
1+\epsilon(k+1,d_1,\dots,d_n)\ge \lambda(g,L_0+1) ,
\end{gather}
where
\begin{gather}
k+d_1+d_2+\dots+d_{n-2}=L_0 ,\qquad n\ge 2,\qquad k\ge 0 ,\nonumber\\
 k+1\le\min_{\substack{1\le i\le n-2\\d_i>0}} d_i ,\qquad g>L_0+1\ge 1 .\label{eq:restrictions}
\end{gather}

We consider the special case $k=0$ separately. In this special case,
when $n=2$ the
partition $(k+1,d_1,\dots,d_n)$ becomes $(1,d_1,d_2)$ and the
desired inequality is proved in~\eqref{eq:epsilon:L:1}. Assume that
$k=0$ and $n\ge 3$. By~\eqref{eq:delta:dilaton} we
have
\begin{gather}\label{eq:delta:dilaton:positive}
\delta_{\rm dilaton}(1,d_1,\dots,d_n)\ge 0 ,\qquad\text{for }n\ge 3 .
\end{gather}
Thus, for any $g> L_0+1$ we have
\begin{gather*}
\big(1+\epsilon(1,d_1,\dots,d_n)\big)\cdot \lfloor \tau_1\tau_{d_1} \cdots \tau_{d_n}\rfloor_{g,n+1}=
\langle \tau_1\tau_{d_1} \cdots \tau_{d_n}\rangle_{g,n+1}\\
\qquad{} =(2g-2+n)\langle \tau_{d_1} \cdots \tau_{d_n}\rangle_{g,n}
\ge (2g-2+n)\cdot\big(\lambda(g,L_0)\cdot
\lfloor \tau_{d_1} \cdots \tau_{d_n}\rfloor_{g,n}\big)
\\
\qquad{} =\lambda(g,L_0)\cdot
\lfloor \tau_1\tau_{d_1} \cdots \tau_{d_n}\rfloor_{g,n+1}
\cdot\big(1+\delta_{\rm dilaton}(1,\boldsymbol{d})\big)
\\
\qquad{} \ge \lambda(g,L_0)\cdot
\lfloor \tau_1\tau_{d_1} \cdots \tau_{d_n}\rfloor_{g,n+1}
> \lambda(g,L_0+1)
\cdot
\lfloor \tau_1\tau_{d_1} \cdots \tau_{d_n}\rfloor_{g,n+1} .
\end{gather*}
Here the first equality is
definition~\eqref{eq:def:of:epsilon}
of $\epsilon(1,d_1,\dots,d_n)$; the second equality
is the string equation~\eqref{eq:string}; the
inequality in the middle of the second line is the
induction assumption; the equality in the beginning of the
third line is definition~\eqref{eq:dilaton:floor}
of $(1+\delta_{\rm dilaton}(1,\boldsymbol{d}))$;
the inequality in the beginning of the last line
is a direct implication of~\eqref{eq:delta:dilaton:positive};
the last inequality is an implication of~\eqref{eq:lambda:inequalities}.

Suppose now that $k\ge 1$. We first prove the desired inequality~\eqref{eq:desired:inequality} in the special case when
\begin{gather}\label{eq:all:dj:strictly:positive}
d_j\ge 1 ,\qquad\text{for }j=1,\dots,n-2
\end{gather}
and then prove it in the most general situation when some of $d_j$ (possibly all of them) are equal to $0$.

Note that by assumption, $k+1$ is less than or equal
to any strictly positive element among $d_1,\dots,d_{n-2}$,
so inequalities~\eqref{eq:all:dj:strictly:positive},
actually, imply that
\begin{gather*}%\label{eq:all:dj:strictly:positive:bis}
d_j\ge k+1 ,\qquad\text{for }j=1,\dots,n-2.
\end{gather*}
Also, from~\eqref{eq:restrictions} we get
\[
k+d_1+d_2+\dots+d_{n-2}=L_0<\frac{3g}{2} .
\]
Thus, the partition $(k+1,d_1,\dots,d_n)$
satisfies assumptions of Corollary~\ref{cor:delat:Virasoro:bound}.

Let $g> L_0+1$. Under the above assumptions we have
\begin{gather*}
\big(1+\epsilon(k+1,d_1,\dots,d_n)\big)\cdot\lfloor\tau_{k+1}\tau_{d_1}\cdots\tau_{d_n}\rfloor_g
=\langle\tau_{k+1}\tau_{d_1}\cdots\tau_{d_n}\rangle_g
 \\
 \qquad{} \ge
\frac{1}{(2k+3)!!}
\Bigg(\sum_{j=1}^n
\frac{(2k+2d_j+1)!!}{(2d_j-1)!!}\langle\tau_{d_1}\cdots
\tau_{d_{j}+k}\cdots\tau_{d_n}\rangle_g\\
\qquad\quad{} +\frac{1}{2}\sum_{\substack{r+s=k-1\\r,s\ge 0}}
(2r+1)!!(2s+1)!!\langle\tau_r\tau_s\tau_{d_1}\cdots\tau_{d_n}\rangle_{g-1}
\Bigg)
\\ \qquad{} \ge
\frac{1}{(2k+3)!!}
\Bigg(\sum_{j=1}^{n-2}
\frac{(2k+2d_j+1)!!}{(2d_j-1)!!}
\cdot\lambda(g,L_0)\cdot
\lfloor\tau_{d_1}\cdots
\tau_{d_{j}+k}\cdots\tau_{d_n}\rfloor_g
\\
\qquad\quad {}+\frac{(2k+2d_{n-1}+1)!!}{(2d_{n-1}-1)!!}
\cdot\lambda(g,L_0-k)\cdot
\lfloor\tau_{d_1}\cdots\tau_{d_{n-2}}
\tau_{d_{n-1}+k}\tau_{d_n}\rfloor_g
\\
\qquad\quad {}+\frac{(2k+2d_{n}+1)!!}{(2d_{n}-1)!!}
\cdot\lambda(g,L_0-k)\cdot
\lfloor\tau_{d_1}\cdots\tau_{d_{n-1}}
\tau_{d_{n}+k}\rfloor_g
\\
\qquad\quad {}+\frac{1}{2}\sum_{\substack{r+s=k-1\\r,s\ge 0}}
(2r+1)!!(2s+1)!!
\cdot\lambda(g-1,L_0-1)\cdot
\lfloor\tau_r\tau_s\tau_{d_1}\cdots\tau_{d_n}\rfloor_{g-1}
\Bigg)
\\
\qquad{} \ge\lambda(g,L_0)\cdot\lfloor\tau_{k+1}\tau_{d_1}\cdots\tau_{d_n}\rfloor_g
\big(1+\delta_{\rm Virasoro}(k+1,\boldsymbol{d})\big)
\\
\qquad{} \ge\left(1-\frac{1}{6g+1}\right)\cdot\lambda(g,L_0)\cdot\lfloor\tau_{k+1}\tau_{d_1}\cdots\tau_{d_n}\rfloor_g
\\
\qquad{} \ge\lambda(g,L_0+1)\cdot\lfloor\tau_{k+1}\tau_{d_1}\cdots\tau_{d_n}\rfloor_g .
\end{gather*}
Here the first equality is
the definition~\eqref{eq:def:of:epsilon}
of $\epsilon(1,d_1,\dots,d_n)$;
the first inequality
is an instant corollary of the
Virasoro constraints in which we omitted
the terms in the third line of~\eqref{eq:virasoro}.
The second inequality is the induction assumption.
The third inequality combines
the inequality $\lambda(g,L_0-k)>\lambda(g,L_0)$
which follows from~\eqref{eq:lambda:inequalities},
the inequality $\lambda(g-1,L_0-1)>\lambda(g,L_0)$
which follows from~\eqref{eq:lambda:inequality:bis},
and the definition~\eqref{eq:Virasoro:floor}
of $\delta_{\rm Virasoro}(k+1,\boldsymbol{d})$.
The inequality
$\big(1+\delta_{\rm Virasoro}(k+1,\boldsymbol{d})\big)\ge
\big(1-\frac{1}{6g+1}\big)$ is justified by~\eqref{eq:delat:Virasoro:bound}.
The last inequality is justified in~\eqref{eq:lambda:inequality:ters}.

It remains to prove inequality~\eqref{eq:desired:inequality},
without extra assumptions~\eqref{eq:all:dj:strictly:positive}.
In other words, we have to prove the inequality
\[
1+\epsilon\big(0^s,k+1,d_1,\dots,d_{n-s}\big)\ge \lambda(g,L_0+1) .
\]

The case $n-s=0$ follows from~\eqref{eq:epsilon:0:1corr}. For $n-s=1$ inequality~\eqref{eq:epsilon:0:2corr} implies
\[
1+\epsilon\big(0^s,k+1,d_1\big)\ge 1-\frac{2}{6g-1} =\lambda(g,0) \ge \lambda(g,L_0+1) .
\]
Thus, we may assume that $n-s\ge 2$ and that the following inequalities are valid:
\[
\begin{cases}
k+d_1+d_2+\dots+d_{n-s-2}=L_0 ,
\\ n\ge s\ge 0 ,
\\ k\ge 1 ,
\\ d_j> k\quad\text{for }j=1,\dots,n-s-2 .
\end{cases}
\]

We proceed by induction in $s$. For $s=0$, which serves us as a base of induction, the statement is already proved. We
perform a step of induction as follows
\begin{gather*}
\big(1+\epsilon\big(0^{s+1},k+1,d_1,\dots,d_{n-s}\big)\big)
\cdot\big\lfloor\tau_0^{s+1}\tau_{k+1}\tau_{d_1}\cdots\tau_{d_{n-s}}\big\rfloor_g
 =\big\langle\tau_0^{s+1}\tau_{k+1}\tau_{d_1}\cdots\tau_{d_{n-s}}\big\rangle_g
\\
\qquad{} =\big\langle\tau_0^s\tau_k\tau_{d_1}\cdots\tau_{d_{n-s}}\big\rangle_g
+\big\langle\tau_0^s\tau_{k+1}\tau_{d_1-1}\cdots\tau_{d_{n-s}}\big\rangle_g
+\dots
+\big\langle\tau_0^s\tau_{k+1}\tau_{d_1}\cdots\tau_{d_{n-s}-1}\big\rangle_g
\\
\qquad{} =\big(1+\epsilon\big(0^s,k,d_1,\dots,d_{n-s}\big)\big)
\cdot\big\lfloor\tau_0^s\tau_k\tau_{d_1}\cdots\tau_{d_{n-s}}\big\rfloor_g
\\
\qquad\quad{}+\big(1+\epsilon\big(0^s,k+1,d_1-1,\dots,d_{n-s}\big)\big)
\cdot\big\lfloor\tau_0^s\tau_{k+1}\tau_{d_1-1}\cdots\tau_{d_{n-s}}\big\rfloor_g
+\cdots
\\
\qquad\quad{} +\big(1+\epsilon\big(0^s,k+1,d_1,\dots,d_{n-s}-1\big)\big)
\cdot\big\lfloor\tau_0^s\tau_{k+1}\tau_{d_1}\cdots\tau_{d_{n-s}-1}\big\rfloor_g
\\
\qquad{} \ge\lambda(g,L_0)
\cdot\big\lfloor\tau_0^s\tau_k\tau_{d_1}\cdots\tau_{d_{n-s}}\big\rfloor_g
+\lambda(g,L_0)
\cdot\big\lfloor\tau_0^s\tau_{k+1}\tau_{d_1-1}\cdots\tau_{d_{n-s}}\big\rfloor_g+\cdots
\\
\qquad\quad{} +\lambda(g,L_0)
\cdot\big\lfloor\tau_0^s\tau_{k+1}\tau_{d_1}\cdots\tau_{d_{n-s-2}-1}\tau_{d_{n-s-1}}\tau_{d_{n-s}}\big\rfloor_g\\
\qquad\quad{} +\lambda(g,L_0+1)
\cdot\big\lfloor\tau_0^s\tau_{k+1}\tau_{d_1}\cdots\tau_{d_{n-s-2}}\tau_{d_{n-s-1}-1}\tau_{d_{n-s}}\big\rfloor_g
\\
\qquad\quad{}
+\lambda(g,L_0+1)
\cdot\big\lfloor\tau_0^s\tau_{k+1}\tau_{d_1}\cdots\tau_{d_{n-s-2}}\tau_{d_{n-s-1}}\tau_{d_{n-s}-1}\big\rfloor_g
\\
\qquad{} \ge\lambda(g,L_0+1)\big(
\big\lfloor\tau_0^s\tau_k\tau_{d_1}\cdots\tau_{d_{n-s}}\big\rfloor_g
+\big\lfloor\tau_0^s\tau_{k+1}\tau_{d_1-1}\cdots\tau_{d_{n-s}}\big\rfloor_g
+\cdots\\
\qquad\quad{}
+\big\lfloor\tau_0^s\tau_{k+1}\tau_{d_1}\cdots\tau_{d_{n-s}-1}\big\rfloor_g
\big)\\
\qquad{} =
\lambda(g,L_0+1)
\cdot\big\lfloor\tau_0^{s+1}\tau_{k+1}\tau_{d_1}\cdots\tau_{d_{n-s}}\big\rfloor_g
\cdot\big(1+\delta_{\rm string}\big(0^{s+1},k+1,d_1,\dots,d_{n-s}\big)\big)\\
\qquad{} \ge \lambda(g,L_0+1)
\cdot\big\lfloor\tau_0^{s+1}\tau_{k+1}\tau_{d_1}\cdots\tau_{d_{n-s}}\big\rfloor_g.
\end{gather*}
Here the first equality is the definition~\eqref{eq:def:of:epsilon} of $\epsilon\big(0^{s+1},k+1,d_1,\dots,d_{n-s}\big)$. The second equality is the string equation~\eqref{eq:string}.
(Recall that by convention, if one of $d_{n-s-1}$ or $d_{n-s}$ is equal to zero, the term, containing the negative index $d_{n-s-1}-1$ or $d_{n-s}-1$ respectively, is missing in the string equation and below.)
The equality which follows, is equation~\eqref{eq:def:of:epsilon} applied to every term of the resulting expression. The inequality, where $\lambda$ appears on the left-hand side for the first time, is the induction assumption applied to each term. The next inequality follows from the inequality $\lambda(g,L_0)>\lambda(g,L_0+1)$,
see~\eqref{eq:lambda:inequalities}. The equality which follows is the definition~\eqref{eq:string:floor} of $\delta_{\rm string}\big(0^{s+1},k+1,d_1,\dots,d_{n-s}\big)$. The last inequality is justified by~\eqref{eq:delta:string:is:positive}.
\end{proof}

\begin{proof}[Proof of Theorem~\ref{th:Main:Theorem}]
For $L=0$ Theorem~\ref{th:Main:Theorem} follows from
Corollary~\ref{cor:L:equals:0}. For $L=1$ it follows from~\eqref{eq:epsilon:L:1}
combined with the fact that $\big(1-\frac{1}{6g+1}\big)
\cdot\big(1-\frac{2}{6g-1}\big)>\lambda(g,1)$.
For $L>1$ we apply recursively Proposition~\ref{eq:prop:step:of:induction}.
\end{proof}

\subsection*{Acknowledgements}
We thank A.~Aggarwal for precious comments on the
preliminary version of this text, which allowed to correct
several misprints and improve the presentation.
We are grateful to N.~Anantharaman, A.~Borodin, D.~Chen,
M.~Kazarian, S.~Lando, M.~M\"oller, A.~Okounkov, M.~Shapiro
for stimulating discussions. We thank MPIM in Bonn and MSRI
in Berkeley for providing us with excellent working
environment.
We thank anonymous referees for their
generous reports and for helpful suggestions which allowed to
improve the presentation.
The research of the second author was partially supported
by PEPS. The results of Section~\ref{s:Introduction} were
obtained at Saint Petersburg State University under support
of RSF grant 19-71-30002. This material is based upon work
supported by the ANR-19-CE40-0021 grant. It was also
supported by the NSF Grant DMS-1440140 while part of the
authors were in residence at the MSRI during the Fall 2019
semester.

\pdfbookmark[1]{References}{ref}
\LastPageEnding


\begin{thebibliography}{99}
\footnotesize\itemsep=0pt

\bibitem{Aggarwal:Siegel:Veech}
Aggarwal A., Large genus asymptotics for {S}iegel--{V}eech constants,
 \href{https://doi.org/10.1007/s00039-019-00509-0}{\textit{Geom. Funct. Anal.}} \textbf{29} (2019), 1295--1324,
 \href{https://arxiv.org/abs/1810.05227}{arXiv:1810.05227}.

\bibitem{Aggarwal:Volumes}
Aggarwal A., Large genus asymptotics for volumes of strata of abelian
 differentials, \href{https://doi.org/10.1090/jams/947}{\textit{J.~Amer. Math. Soc.}}, {t}o appear, \href{https://arxiv.org/abs/1804.05431}{arXiv:1804.05431}.

\bibitem{Aggarwal:intersection:numbers}
Aggarwal A., Large genus asymptotics for intersection numbers and principal
 strata volumes of quadratic differentials, \href{https://arxiv.org/abs/2004.05042}{arXiv:2004.05042}.

\bibitem{ADGZZ:conjecture}
Aggarwal A., Delecroix V., Goujard \'E., Zograf P., Zorich A., Conjectural large
 genus asymptotics of {M}asur--{V}eech volumes and of area {S}iegel--{V}eech
 constants of strata of quadratic differentials, \href{https://doi.org/10.1007/s40598-020-00139-7}{\textit{Arnold Math.~J.}}
 \textbf{6} (2020), 149--161, \href{https://arxiv.org/abs/1912.11702}{arXiv:1912.11702}.

\bibitem{Chen:Moeller:Sauvaget}
Chen D., M\"oller M., Sauvaget A., Masur--{V}eech volumes and intersection
 theory: the principal strata of quadratic differentials (with an appendix by
 G.~Borot, A.~Giacchetto, D.~Lewanski), \href{https://arxiv.org/abs/1912.02267}{arXiv:1912.02267}.

\bibitem{Chen:Moeller:Sauvaget:Zagier}
Chen D., M\"oller M., Sauvaget A., Zagier D., Masur--{V}eech volumes and
 intersection theory on moduli spaces of abelian differentials,
 \href{https://doi.org/10.1007/s00222-020-00969-4}{\textit{Invent. Math.}}, {t}o appear, \href{https://arxiv.org/abs/1901.01785}{arXiv:1901.01785}.

\bibitem{DGZZ:volume}
Delecroix V., Goujard \'E., Zograf P., Zorich A., Masur--{V}eech volumes,
 frequencies of simple closed geodesics and intersection numbers of moduli
 spaces of curves, \href{https://arxiv.org/abs/1908.08611}{arXiv:1908.08611}.

\bibitem{DGZZ:random:multicurves}
Delecroix V., Goujard \'E., Zograf P., Zorich A., Large genus asymptotic geometry
 of random square-tiled surfaces and of random multicurves,
 \href{https://arxiv.org/abs/2007.04740}{arXiv:2007.04740}.

\bibitem{Dij}
Dijkgraaf R., Intersection theory, integrable hierarchies and topological field
 theory, in New Symmetry Principles in Quantum Field Theory ({C}arg\`ese,
 1991), \textit{NATO Adv. Sci. Inst. Ser. B Phys.}, Vol.~295, Plenum, New
 York, 1992, 95--158, \href{https://arxiv.org/abs/hep-th/9201003}{arXiv:hep-th/9201003}.

\bibitem{DVV}
Dijkgraaf R., Verlinde H., Verlinde E., Topological strings in {$d<1$},
 \href{https://doi.org/10.1016/0550-3213(91)90129-L}{\textit{Nuclear Phys.~B}} \textbf{352} (1991), 59--86.

\bibitem{Eskin:Okounkov:pillowcase}
Eskin A., Okounkov A., Pillowcases and quasimodular forms, in Algebraic
 Geometry and Number Theory, \textit{Progr. Math.}, Vol.~253, \href{https://doi.org/10.1007/978-0-8176-4532-8_1}{Birkh\"auser
 Boston}, Boston, MA, 2006, 1--25, \href{https://arxiv.org/abs/math.DS/0505545}{arXiv:math.DS/0505545}.

\bibitem{Eskin:Zorich}
Eskin A., Zorich A., Volumes of strata of {A}belian differentials and
 {S}iegel--{V}eech constants in large genera, \href{https://doi.org/10.1007/s40598-015-0028-0}{\textit{Arnold Math.~J.}}
 \textbf{1} (2015), 481--488, \href{https://arxiv.org/abs/1507.05296}{arXiv:1507.05296}.

\bibitem{Goujard:volumes}
Goujard \'E., Volumes of strata of moduli spaces of quadratic differentials:
 getting explicit values, \href{https://doi.org/10.5802/aif.3062}{\textit{Ann. Inst. Fourier (Grenoble)}} \textbf{66}
 (2016), 2203--2251, \href{https://arxiv.org/abs/1501.01611}{arXiv:1501.01611}.

\bibitem{Kazarian}
Kazarian M., Recursion for {M}asur--{V}eech volumes of moduli spaces of
 quadratic differentials, \textit{J.~Inst. Math. Jussieu}, {t}o appear,
 \href{https://arxiv.org/abs/1912.10422}{arXiv:1912.10422}.

\bibitem{Kazarian:Lando}
Kazarian M., Lando S., An algebro-geometric proof of {W}itten's conjecture,
 \href{https://doi.org/10.1090/S0894-0347-07-00566-8}{\textit{J.~Amer. Math. Soc.}} \textbf{20} (2007), 1079--1089,
 \href{https://arxiv.org/abs/math.AG/0601760}{arXiv:math.AG/0601760}.

\bibitem{Kontsevich}
Kontsevich M., Intersection theory on the moduli space of curves and the matrix
 {A}iry function, \href{https://doi.org/10.1007/BF02099526}{\textit{Comm. Math. Phys.}} \textbf{147} (1992), 1--23.

\bibitem{Liu:Xu}
Liu K., Xu H., A remark on {M}irzakhani's asymptotic formulae, \href{https://doi.org/10.4310/AJM.2014.v18.n1.a2}{\textit{Asian~J.
 Math.}} \textbf{18} (2014), 29--52, \href{https://arxiv.org/abs/1103.5136}{arXiv:1103.5136}.

\bibitem{Mirzakhani:simple:geodesics:and:volumes}
Mirzakhani M., Simple geodesics and {W}eil--{P}etersson volumes of moduli
 spaces of bordered {R}iemann surfaces, \href{https://doi.org/10.1007/s00222-006-0013-2}{\textit{Invent. Math.}} \textbf{167}
 (2007), 179--222.

\bibitem{Mirzakhani:volumes:and:intersection:theory}
Mirzakhani M., Weil--{P}etersson volumes and intersection theory on the moduli
 space of curves, \href{https://doi.org/10.1090/S0894-0347-06-00526-1}{\textit{J.~Amer. Math. Soc.}} \textbf{20} (2007), 1--23.

\bibitem{Okounkov:Pandharipande}
Okounkov A., Pandharipande R., Gromov--{W}itten theory, {H}urwitz numbers, and
 matrix models, in Algebraic Geometry~-- {S}eattle 2005, {P}art~1,
 \href{https://doi.org/10.1090/pspum/080.1/2483941}{\textit{Proc. Sympos. Pure Math.}}, Vol.~80, Amer. Math. Soc., Providence, RI,
 2009, 325--414, \href{https://arxiv.org/abs/math.AG/0101147}{arXiv:math.AG/0101147}.

\bibitem{Witten}
Witten E., Two-dimensional gravity and intersection theory on moduli space, in
 Surveys in Differential Geometry ({C}ambridge, {MA}, 1990), \href{https://dx.doi.org/10.4310/SDG.1990.v1.n1.a5}{Lehigh
 University}, Bethlehem, PA, 1991, 243--310.

\bibitem{Yang:Zagier:Zhang}
Yang D., Zagier D., Zhang Y., Masur--Veech volumes of quadratic differentials
 and their asymptotics, \href{https://arxiv.org/abs/2005.02275}{arXiv:2005.02275}.

\bibitem{Zograf:2:correlators}
Zograf P.G., An explicit formula for {W}itten's 2-correlators, \href{https://doi.org/10.1007/s10958-019-04371-1}{\textit{J.~Math.
 Sci.}} \textbf{240} (2019), 535--538.

\end{thebibliography}
\end{document}